\documentclass[12pt]{amsart}
\usepackage{amsfonts,amssymb,latexsym,amsmath, amsxtra,verbatim}
\usepackage[all]{xy}
\usepackage{graphicx}

\usepackage[OT2,T1]{fontenc}
\DeclareSymbolFont{cyrletters}{OT2}{wncyr}{m}{n}
\DeclareMathSymbol{\Sha}{\mathalpha}{cyrletters}{"58}

\theoremstyle{plain}
\newtheorem{theorem}{Theorem}[section]

\newtheorem{lemma}[theorem]{Lemma}

\theoremstyle{definition}

\theoremstyle{remark}

\numberwithin{equation}{section}

\newcommand{\F}{{\mathbb F}}

\def\({\left(}
\def\){\right)}
\def\<{\left<}
\def\>{\right>}

\newcommand{\wt}[1]{\widetilde{#1}}

\newcommand{\abs}[1]{\left|#1\right|}

\def\E{{\mathbb E}}

\def\Pars{\Pi}

\def\DES{{\rm DES}}
\def\VAR{{\rm VAR}}
\def\Prob{{\bf P}}

\begin{document}

\title[central limit theorems for some set partition statistics ]
{
Central limit theorems for some set partition statistics
}
\author{Bobbie Chern}
\address{Stanford University, Department of Electrical Engineering, Stanford, CA 94305}
\email{bgchern@stanford.edu}
\author{Persi Diaconis}
\address{Stanford University, Department of Mathematics and Statistics, Sequoia Hall, 390 Serra Mall, Stanford, CA 94305-4065, USA}
\email{diaconis@math.stanford.edu}
\author{Daniel M. Kane}
\address{University of California, San Diego, Department of Mathematics, 9500 Gilman Drive \#0404, La Jolla, CA 92093}
\email{dakane@ucsd.edu}
\author{Robert C. Rhoades}
\address{Center for Communications Research, Princeton, NJ 08540}
\email{rob.rhoades@gmail.com}

\thanks{}

\date{\today}
\thispagestyle{empty} \vspace{.5cm}
\begin{abstract}
We prove the conjectured limiting normality for the number of crossings of a
uniformly chosen set partition of $[n] = \{ 1, 2, \ldots, n\}$.  The arguments
use a novel stochastic representation and are also used to prove central limit
theorems for the dimension index and the number of levels.
\end{abstract}


\maketitle

\section{Introduction}
Let $\lambda$ be a partition of the set $[n] = \{1, 2, \ldots, n\}$, so $1|2|3,
12|3, 13|2, 1|23, 123$ are the five partitions of $[3]$.  The enumerative
theory of ``supercharacters'' leads to the statistics
\begin{equation}\label{eqn:defs}
d(\lambda) = \sum_i ( M_i - m_i + 1) \ \ \text{ and } \   \ cr(\lambda) = \# \text{ of crossings of } \lambda.
\end{equation}
In $d(\lambda)$, the sum is over the blocks of $\lambda$ and $M_i$ ($m_i$) is
the largest (smallest) element of the block $i$. The statistic $cr(\lambda)$
counts $i < i' < j < j'$ with $i,j$ adjacent elements of the same block and $i', j'$ adjacent elements of the same
block (\includegraphics[scale=0.25]{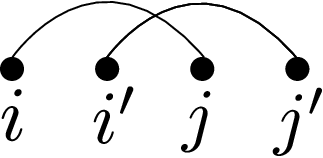}).  In a companion
paper \cite{CDKR} the moments of $d(\lambda)$ and $cr(\lambda)$ are determined
as explicit linear combinations of Bell numbers $B_n$.  Numerical computations
(see  Figures \ref{fig:dimensionHistograms} and
\ref{fig:intertwiningHistograms}) suggests that normalized by their mean and
variance, these statistics have approximate normal distributions.
Figures \ref{fig:dimensionHistograms} -- \ref{fig:levelHistograms} are based on
exact counts from our new algorithms \cite{CDKR}. Figures \ref{fig:dimensionHistograms} and \ref{fig:intertwiningHistograms}
suggest good
agreement with the normal approximation for dimension index and
crossings. Figure 3 shows slower convergence for levels and suggests
a search for finite sample correction terms.
We found the limiting normality challenging to prove using available techniques (eg. moments, Fristedt's
method of conditioned limit theorems \cite{fristedt}, or Stein's method
\cite{CGS}). Indeed, the limiting normality of $cr(\lambda)$ is conjectured in
\cite{kasraoui}.

\begin{figure}[h!]
\center
\includegraphics[scale=0.45]{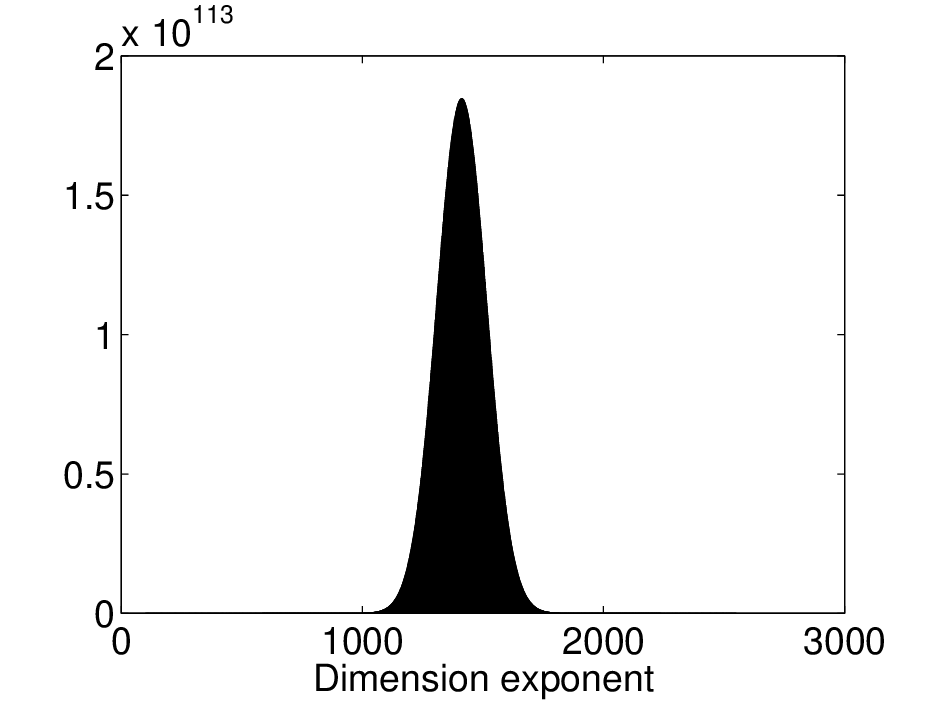}
\includegraphics[scale=0.45]{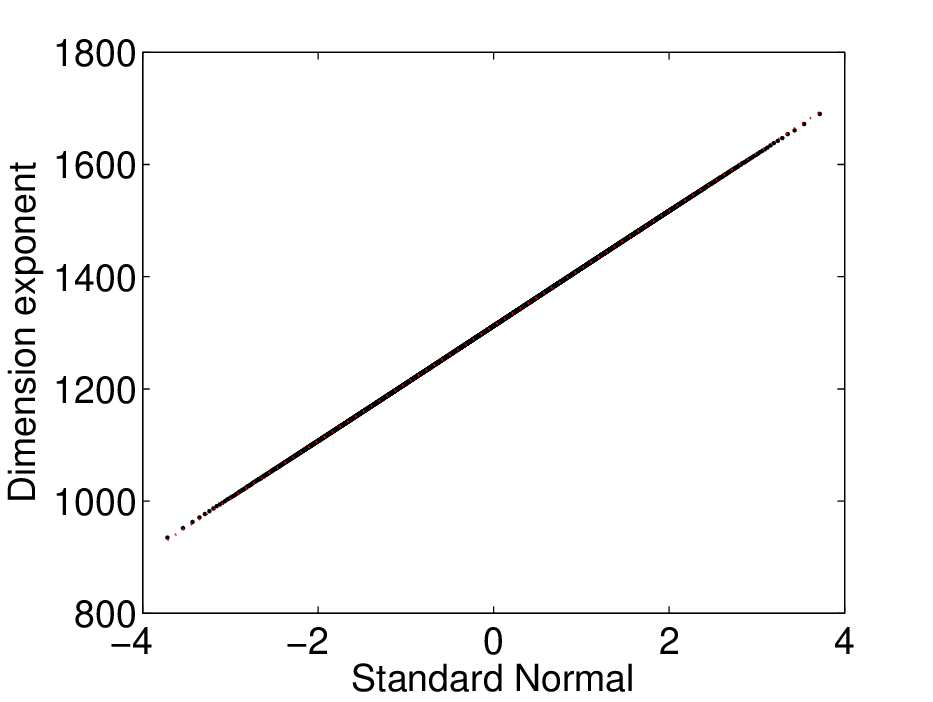}
\caption{Histogram of the dimension exponent counts for $n=100$ and the associated Q-Q plot.}
\label{fig:dimensionHistograms}
\end{figure}

\begin{figure}[h!]
\center
\includegraphics[scale=0.45]{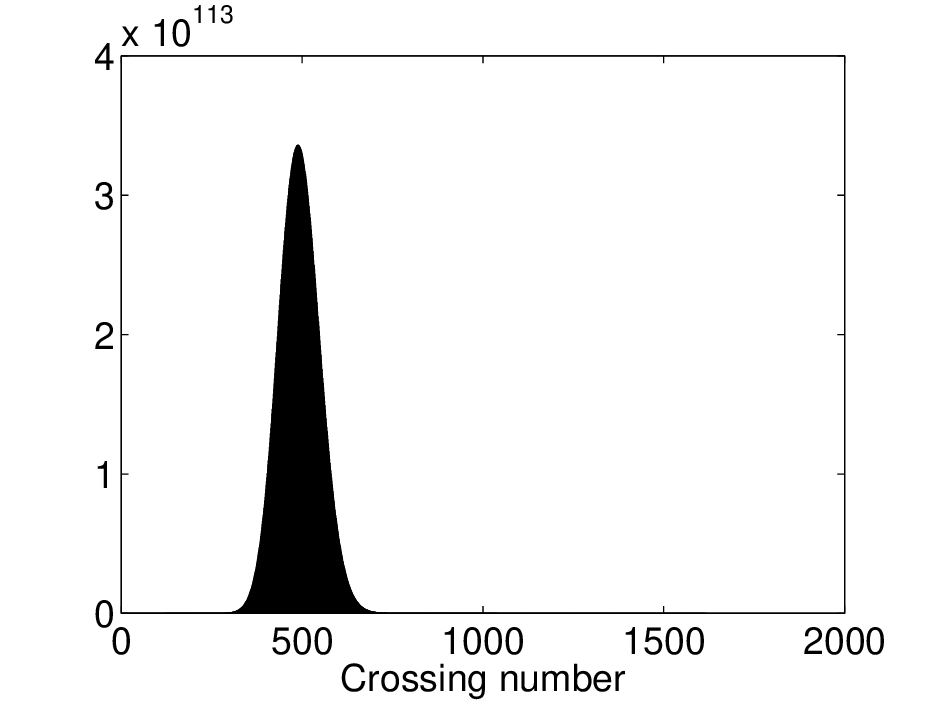}
\includegraphics[scale=0.45]{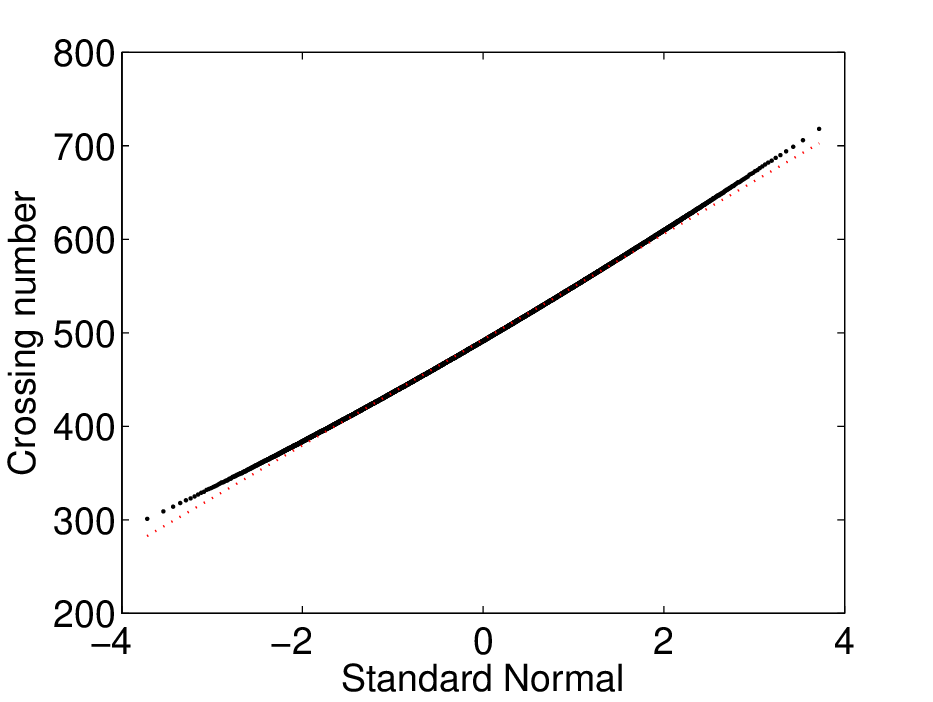}
\caption{Histogram of the crossing number counts for $n=100$ and the associated Q-Q plot.}
\label{fig:intertwiningHistograms}
\end{figure}

\begin{figure}[h!]
\center
\includegraphics[scale=0.45]{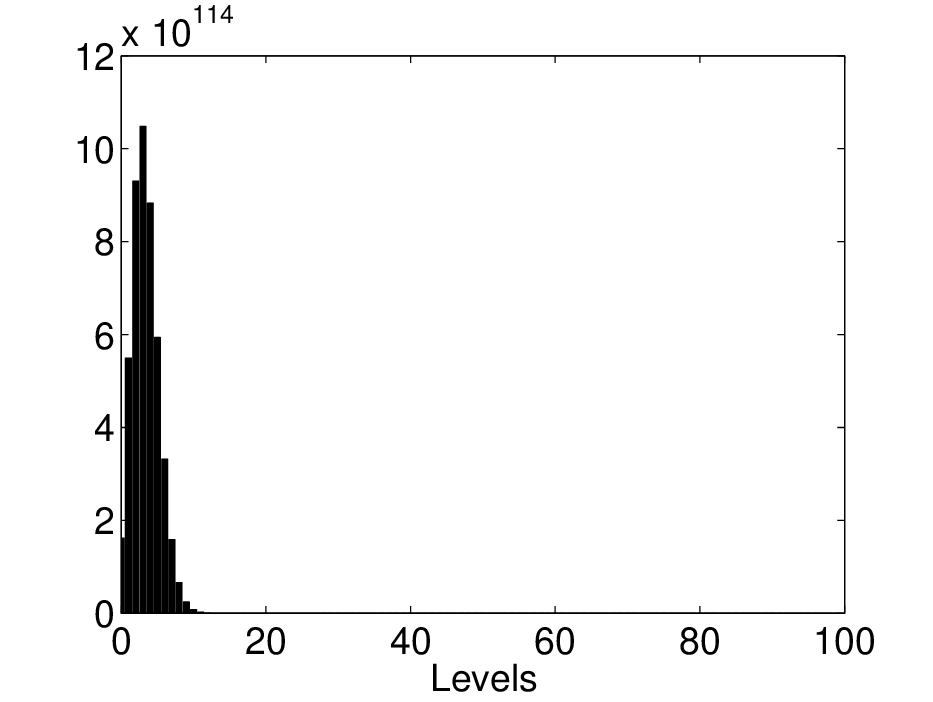}
\includegraphics[scale=0.45]{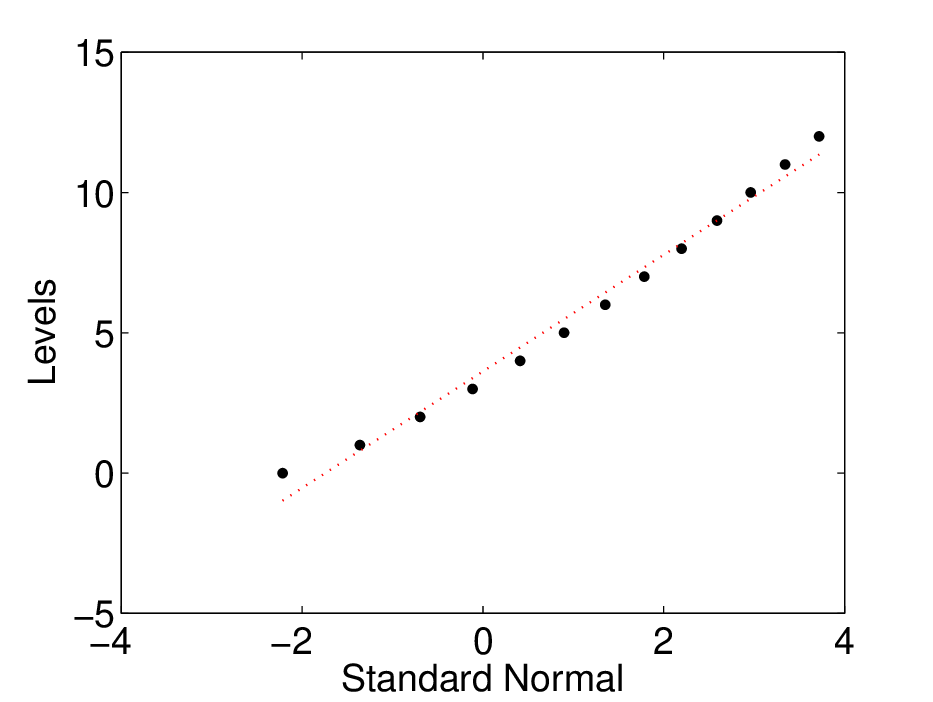}
\caption{Histogram of the level counts for $n=100$ and the associated Q-Q plot.}
\label{fig:levelHistograms}
\end{figure}


A key ingredient of the present paper is a stochastic algorithm for generating
a random set partition due to Stam \cite{stam}. Supplementing this with some
novel probabilistic ideas allows standard ``delta method'' techniques to finish
the job.

Brief reviews of the extensive enumerative, algebraic and probabilistic aspects
of set partitions are in \cite{knuth4a} and \cite{stanley1}.
The book of Mansour \cite{mansour} contains applications to computer science
and much else.  An important paper combining many of the statistics we
work with is \cite{cddsy}.
The companion paper \cite{CDKR} has an extensive review.  It
also summarizes the literature on supercharacters. Briefly, these are natural
characters $\chi_\lambda$ on the uni-upper triangular matrix group $U_n(\F_q)$
which are indexed by set partitions. The representation corresponding to
$\chi_\lambda$ has dimension $q^{d(\lambda)}$.  The (usual) inner product
between $\chi_\lambda$ and $\chi_\mu$ is $<\chi_\lambda, \chi_\mu> =
q^{cr(\lambda)} \delta_{\lambda, \mu}$. This suggests understanding how
$d(\lambda)$ and $cr(\lambda)$ vary for typical set partitions.

There are many codings of a set partition. One needed below codes $\lambda$ as a
sequence $x_1, x_2, \cdots, x_n$ with $x_i = j$ if and only if $i$ is in block
$j$ of $\lambda$. Thus $135|24|6|7$ corresponds to $1, 2, 1, 2, 1, 3, 4$.  If
$a_i = x_i -1$, $a_1, a_2, \cdots , a_n$ is a restricted growth sequence: $a_1
= 0$ and $a_{i+1} \le 1 + \max( a_1, \cdots, a_{i})$ for $1 \le i \le n -1$.
This standard coding is discussed in \cite[page 416]{knuth4a}.  For this coding, let
\begin{equation}\label{eqn:defs2}
L(\lambda) = \abs{ \{ i: x_{i+1} = x_i\}}
\end{equation}
the number of \underline{levels}  of $\lambda$.  This is used
as an example of the present techniques.  See \cite[Chapter 4]{mansour} for further references.

The main theorems proved use $\alpha_n$, the positive real solution of $ue^u =
n+1$ (so $\alpha_n = \log(n) - \log\log(n) + o(1) $ \cite{deBruijn}).  Let
$\Pars(n)$ be the set of partitions of $[n]$.  Throughout, $\lambda$ is
uniformly chosen in $\Pars(n)$.

\begin{theorem}\label{thm:levels} The number of levels $L(\lambda)$ has
$\mu_n^L = \E(L(\lambda)) = (n-1) \frac{B_{n-1}}{B_n}\sim \log(n)$ and
$\( \sigma_n^L\)^2 = \VAR(L(\lambda)) = (n-1) \frac{B_{n-1}}{B_n}
+ n(n-1) \frac{B_{n-2}}{B_{n}} - (n-1)^2 \frac{B_{n-1}^2}{B_n^2} \sim \log(n).$
Normalized by its mean and standard deviation, $L(\lambda)$ has an approximate standard normal distribution
$$\Prob\( \frac{L(\lambda) - \mu_n^L}{\sigma_n^{L}} \le x \) \to \frac{1}{\sqrt{2\pi}} \int_{-\infty}^x e^{-\tau^2/2} d\tau$$
for all fixed $x$ as $n\to \infty$.
\end{theorem}


\begin{theorem}\label{thm:dimension} The dimension index $d(\lambda)$ has
$\mu_n^d = \E(d(\lambda)) = \frac{\alpha_n-2}{\alpha_n} n^2 + O\( \frac{n}{\alpha_n}\)$ and
$\( \sigma_n^d\)^2 = \VAR(d(\lambda)) = \( \frac{\alpha_n^2 - 7\alpha_n + 17}{\alpha_n^3(\alpha_n+1)}\) n^3
+ O\( \frac{n^2}{\alpha_n}\).$
Normalized by its mean and standard deviation, $d(\lambda)$ has an approximate standard normal distribution
$$\Prob\( \frac{d(\lambda) - \mu_n^d}{\sigma_n^{d}} \le x \) \to \frac{1}{\sqrt{2\pi}}  \int_{-\infty}^x e^{-\tau^2/2} d\tau$$
for all fixed $x$ as $n\to \infty$.
\end{theorem}

\begin{theorem}\label{thm:crossing} The number of crossings $cr(\lambda)$ has
$\mu_n^{cr} = \E(cr(\lambda)) = \frac{2\alpha_n-5}{4\alpha_n^2} n^2 + O\( \frac{n}{\alpha_n}\)$ and
$\( \sigma_n^{cr}\)^2 = \VAR(cr(\lambda)) =  \frac{3\alpha_n^2 - 22\alpha_n + 56}{9\alpha_n^3(\alpha_n+1)} n^3
+ O\( \frac{n^2}{\alpha_n}\).$
Normalized by its mean and standard deviation, $cr(\lambda)$ has an approximate standard normal distribution
$$\Prob\( \frac{cr(\lambda) - \mu_n^{cr}}{\sigma_n^{cr}} \le x \) \to \frac{1}{\sqrt{2\pi}}  \int_{-\infty}^x e^{-\tau^2/2} d\tau$$
for all fixed $x$ as $n\to \infty$.
\end{theorem}

Section \ref{sec:stam} of this paper explains Stam's algorithm and shows how it
gives a useful heuristic picture of what a random set partition ``looks like''.
The limit theorem for levels is proved in Section
\ref{sec:levels_descents} as a simple illustration of our proof technique.  The
dimension index and number of crossings require further ideas.  They are given
separate proofs in Sections \ref{sec:dimension} and \ref{sec:crossing}.

\section*{Notation} Throughout, we use the stochastic order symbols $O_p$ and
$o_p$.  If $X_n$ for $1 \le n < \infty$ is a sequence of real valued random
variables and $a_n$ is a sequence of real numbers, write $X_n = O_p(a_n)$ if
for every $\epsilon> 0$ and some $\eta>0$, which may depend on $\epsilon$,
 there is $N$ so that $\Prob\{
\abs{X_n} \le \eta \abs{a_n} \} > 1-\epsilon$ for all $n > N$.  Write $X_n =
o_p(a_n)$ if for every $\epsilon >0$ and $\eta>0$ there is $N$ so that $\Prob\{
\abs{X_n} < \eta \abs{a_n}\} > 1-\epsilon$ for all $n >N$.  For background,
examples and many variations see Pratt \cite{pratt}, Lehman \cite{lehman}, or
Serfling \cite{serfling}.  We say two sequences of random variables are weak star
close if their distributions are close in L\'evy metric.

\section{Stam's algorithm and set partition heuristics}\label{sec:stam}
Write $\Pars(n)$ for the set partitions of $[n] = \{1, 2, \cdots, n\}$ and $B_n
= \abs{\Pars(n)}$ for the $n$th Bell number (sequence A000110 of Sloane's
\cite{sloane}).  To help evaluate asymptotics it is helpful to have
\begin{align*}\label{eqn:bell_asymp}
& \frac{B_{n+1}}{B_n} = \frac{n}{\alpha_n} + \frac{1}{2} \frac{\alpha_n}{(1+\alpha_n)^2}  + O\( \frac{\alpha_n}{n}\) \\
& \frac{B_{n+k}}{B_n} = \frac{(n+k)!}{n! \alpha_n^k}\( 1+ O\( \frac{k}{n \alpha_n}\)\) \\
& \frac{\alpha_{n+k}}{\alpha_n}  =  1 + O\( \frac{k}{n\log(n)}\)
\end{align*}
which are valid for fixed $k$ as $n\to \infty$.
See, for instance, \cite{deBruijn}.  Dobinski's identity
\cite{dobinski, pitman1}
\begin{equation}
B_n = \frac{1}{e} \sum_{m=1}^\infty \frac{m^n}{m!}
\end{equation}
shows that for fixed $n \in \{ 0, 1, 2, \cdots\}$
\begin{equation}\label{eqn:measure}
\mu_n(m) = \frac{1}{e B_n} \frac{m^n}{m!}
\end{equation}
is a probability measure on $\{1, 2, 3, \cdots \}$. Stam \cite{stam} uses this
measure to give an elegant algorithm for choosing a uniform random element of
$\Pars(n)$.

\begin{center}
\underline{Stam's Algorithm}
\end{center}
\begin{enumerate}
\item Choose $M$ from $\mu_n$.
\item Drop $n$ labelled balls uniformly into $M$ boxes.
\item Form a set partition $\lambda$ of $[n]$ with $i$ and $j$ in the same block if and only if balls $i$ and $j$ are
in the same box.
\end{enumerate}
Of course, after choosing $M$ and dropping  balls, some of the boxes may be
empty. Stam \cite{stam} shows that the number of empty boxes has (exactly) a
Poisson distribution and is independent of the generated set partition. This
implies that the number of boxes $M$ drawn from $\mu_n$ at
\eqref{eqn:measure} has the same limiting distribution as the number of blocks
in a random $\lambda \in \Pars(n)$. This is a well studied random variable.  It
will emerge that the fluctuations of $M$ are the main source of randomness in
Theorems \ref{thm:levels} -- \ref{thm:crossing}. Results of Hwang \cite{hwang}
prove the following normal limit theorem (Hwang also has an error estimate).
\begin{theorem}\label{thm:hwang}
For $M$ chosen from $\mu_n$ of \eqref{eqn:measure}, as $n\to\infty$
$$\mu_n^M := \E(M) = \frac{B_{n+1}}{B_n} = \frac{n}{\alpha_n} + O\(\frac{1}{\alpha_n}\)$$
and
$$\(\sigma_n^M\)^2:= {\rm VAR}(M)  = \frac{B_{n+2}}{B_n} - \frac{B_{n+1}^2}{B_n^2} = \frac{n}{\alpha_n^2}  + O\(\frac{n}{\alpha_n^3}\).$$
Normalized by its mean and standard deviation, $M$ has an approximate standard normal distribution.
\end{theorem}

\noindent {\bf Heuristic I.}  Stam's algorithm gives a useful intuitive way to
think about a random element of $\Pars(n)$. It behaves practically the same as
a uniform multinomial allocation of $n$ labelled balls into $m = n/\log(n)$
boxes. The arguments in the following sections make this precise.  It appears
to us that many of the features previously treated in the beautiful paper of
Fristedt \cite{fristedt} can be treated by the present approach.  Note that
Fristedt treated features that only depend on block sizes (largest, smallest,
number of boxes of size $i$).  None of our statistics have this form.

\noindent {\bf Heuristic II.}  Fristedt's arguments randomize $n$. This makes
the block variables, $N_i(\lambda) = \# \text{ blocks of size } i$, independent
allowing standard probability theorems to be used.  At the end, a Tauberian
argument (dePoissonization) is used to show that the theorems hold for fixed
$n$.  The present argument fixes $n$ and randomizes the number of blocks.  This
results in a ``balls in boxes'' problem with many tools available. At the end,
an Abelian argument shows that the appropriate limit theorem holds when $m$
fluctuates.  See \cite{hardy} for background on this use of Abelian and
Tauberian theorems.  There are many variants of Poissonization in active use.
We do not see how to abstract Stam's algorithm to other combinatorial
structures.

We conclude this section with a simple illustration of Stam's algorithm.  From
\eqref{eqn:measure}, $\mu_n(m) = \frac{1}{eB_n} \frac{m^n}{m!}$  is a
probability measure on $\{ 1, 2, 3, \cdots \}$.  Thus for $-n< d < \infty$
\begin{equation}
\E_n(M^d) = \frac{1}{e B_n} \sum_{m=1}^\infty \frac{m^{n+d}}{m!} = \frac{B_{n+d}}{B_n}.
\end{equation}
Let us apply this to compute the moments for $L(\lambda)$, the number of levels
of $\lambda \in \Pars(n)$.  From the definition \eqref{eqn:defs2},  given $M$,
$L(\lambda) = X_1 + \cdots + X_{n-1}$ where
$X_i$ is the indicator random variable of the event that balls $i$ and $i+1$ are dropped into the same box. By
inspection, the $X_i$ are independent with $\Prob(X_i = 1) = \frac{1}{M}$.
Thus
\begin{equation}\label{eqn:levels_mean}
\E_n (L(\lambda)) = \E_n \E \( L(\lambda | M) \) = \E_n \( \frac{n-1}{M} \)  = (n-1) \frac{B_{n-1}}{B_n}.
\end{equation}
The standard identity
$$\VAR(Z) = \E(\VAR(Z|W)) + \VAR(\E(Z|W))$$
for any random variables $Z$ and $W$ such that the moments exist, shows that
\begin{equation}\label{eqn:levels_variance}
\VAR_n(L(\lambda)) = (n-1) \frac{B_{n-1}}{B_n} + n(n-1) \frac{B_{n-2}}{B_{n}} - (n-1)^2 \frac{B_{n-1}^2}{B_n^2}.
\end{equation}

More generally, this provides an alternative approach to \cite{CDKR} for
showing that the moments of statistics $T(\lambda)$ are shifted Bell
polynomials. It requires $\E_n( T(\lambda)  | m)$ to be a Laurent polynomial in
$m$.  As an example, Stam worked with $W_i(\lambda)$, the size of the block in
$\lambda$ containing $i$, $1 \leq i \leq n$.  Then, any polynomial in the
$\{W_i\}_{i=1}^n$ has expectation a shifted Bell polynomial; for example,
$W_i^k$ and $W_i W_j$.  Stam proves that $W_i$ is approximately normal.

\section{Proof of Theorem \ref{thm:levels} }\label{sec:levels_descents}
Theorem \ref{thm:levels} is proved here as a simple illustration of our technique.
Conditioning on $M$ in Stam's algorithm, classical ``balls in bins'' central limit theorems
are used to prove the limiting normality uniformly in $M$ and standard $\delta$-method arguments
are used to complete the proof.

\begin{proof}[Proof of Theorem \ref{thm:levels}]
The moments  of the level statistic $L(\lambda)$ are computed in
\eqref{eqn:levels_mean} and \eqref{eqn:levels_variance}.  Conditional on $M$, $L(\lambda) =
X_1 + \cdots + X_{n-1}$ with $X_i$ independent identically distributed binary variables with
 $\Prob(X_i = 1)  = 1/M$.
Thus conditioned on $M$,
$$\E (L(\lambda) \mid M) = \frac{n-1}{M}, \ \ \ \text{ and } \ \ \ \VAR(L(\lambda) \mid M) = \frac{n-1}{M} \( 1- \frac{1}{M}\).$$
and, normalized by its conditional mean and variance, $L(\lambda)$ has a standard normal limiting
distribution provided $n/M\to \infty$.  In the present case,
$M = M_n$ is a random variable. From Theorem \ref{thm:hwang},
as $n$ tends to infinity
\begin{equation}\label{eqn:normal_limit}
\frac{M_n - \mu_n^M}{\sigma_n^M} \to N(0, 1) \text{ with } \mu_n^M \sim \frac{n}{\alpha_n}, (\sigma_n^M)^2 \sim \frac{n}{\alpha_n^2}.
\end{equation}
This implies
\begin{equation}\label{eqn:n/m_n_1}
\frac{n}{M_n} = \alpha_n + O_p\( \frac{1}{\sqrt{n}}\).
\end{equation}
To be precise,
write $M_n = \mu_n^M + Z_n \sigma_n^M$ with $Z_n = \frac{M_n - \mu_n^M}{\sigma_n^M}$.  Then
\begin{equation}\label{eqn:n/m_n}
\frac{n}{M_n} = \frac{n}{\mu_n^M + Z_n \sigma_n^M} = \frac{n}{\mu_n^M \( 1 + \frac{Z_n}{\sigma_n^M} \mu_n^M\)} = \frac{n}{\mu_n}
  \( 1+ \frac{Z_n}{\mu_n^M} \sigma_n^M + O\( \(\frac{Z_n \sigma_n^M}{\mu_n^M}\)^2 \)\).
\end{equation}
From Theorem \ref{thm:hwang}, $n/\mu_n^M = \alpha_n + O(\alpha_n/n)$, $\sigma_n^M/\mu_n^M = O(1/\sqrt{n})$.  Since
$Z_n = O_p(1)$, \eqref{eqn:n/m_n_1} follows.

Thus, with probability close to 1 with respect to $M$
we have that $L(\lambda)$ conditioned on $M$ is weak star close to a Gaussian with mean
$$
\mu^M =\frac{n-1}{M} = \alpha_n +O_p(n^{-1/2})
$$
and standard deviation
$$
\sigma^M = \sqrt{\frac{n-1}{M}\left(1-\frac{1}{M} \right)} = \sqrt{\alpha_n} + O_p(n^{-1/2}).
$$
Thus, with high probability over $M$, the conditional distribution on $L(\lambda)$ is
weak star close
to $N(\alpha_n,\sqrt{\alpha_n})$. Therefore, the overall distribution of $L(\lambda)$
is also close to this normal distribution.

\end{proof}

\section{Proof of Theorem \ref{thm:dimension}}\label{sec:dimension}
In outline, the proof proceeds by choosing a random $\lambda \in \Pars(n)$
using Stam's algorithm.  Conditioning on the chosen $m$ reduces the problem to
a slightly non-standard balls in boxes problem.  Given $m$, it is shown that
$d(\lambda) = nm -2m^2 + O_p(m^{3/2})$
so that the fluctuations in
$d(\lambda)$ are driven by the fluctuations in $m$.  These are asymptotically
normally distributed with mean and variance $\( \frac{n}{\alpha_n},
\frac{n}{\alpha_n^2}\)$.  From Theorem \ref{thm:hwang} above, a simple
averaging argument completes the proof.  The first proposition treats the balls
in boxes argument.  It proves more than is needed.  The argument is useful for
statistics such as $T(\lambda) = \sum_i M_i$ where the sum runs over the blocks
of $\lambda$ indexed by $i$ and $M_i$ is the maximum element in the $i$th
block.

The first step in the proof is to prove the appropriate approximation
conditional on $m$.  While it would be of interest to explore this for general
$n$, $m$, we content ourselves with proving what is needed for Theorem
\ref{thm:dimension}. From Theorem \ref{thm:hwang} the relevant values of $m$
are $\frac{n}{\alpha_n} + \frac{c \sqrt{n}}{\log(n)}$ for large fixed values of
$c$.  This explains the choice in the next lemma.
\begin{lemma}\label{lem:balls_in_bins}
Fix a large number $C$.
Let $n$ balls labeled $1, 2, \cdots, n$ be dropped uniformly at random into $m$
boxes with $m = \frac{n}{\alpha_n} + \frac{c\sqrt{n}}{\log(n)}$.  For $\abs{c}
\le C$.  Let $$D_n = \sum_{i=1}^m \( M_i - m_i+1\)$$ with $M_i$ the maximum
label in box $i$ and $m_i$ the minimum label of box $i$. $M_i - m_i$ is omitted
if box $i$ is empty.  Then $D_n = nm -2m^2 + O_{p, C}(m^{3/2})$ uniformly in $\abs{c}
\le C$.
\end{lemma}
\begin{proof}
Consider an infinite supply of balls labelled $1, 2, 3, \ldots$ dropped
uniformly at random into $m$ boxes.  Let $W_i$ $1\le i \le m$ be the waiting
time until $i$ boxes have been filled.  Thus $W_i = 1$, $W_2 - W_1$ is
GEOMETRIC($1/m$), $W_3 - W_2$ is GEOMETRIC($2/m$), \ldots, $W_n - W_{n-1}$ is
GEOMETRIC($(m-1)/m$) and all these differences are independent.  Here, if $X$
is GEOMETRIC($\theta$), $\Prob(X = j) = \theta^{j-1}(1-\theta)$, $\E(X) =
1/\theta$, and $\VAR(X) = \frac{1}{\theta} \( \frac{1}{\theta} - 1\)$. Let
$E_t$ be the number of empty boxes at time $t$ and $L_t$ be the largest $\ell$
so that $W_\ell \le t$. If $L_m \le t$ all boxes are non-empty at time $t$ and
$E_t = 0$.  More generally, $L_t = m -E_t$.

The sum $\sum_{i=1}^m m_i$ is $ W_1 + \cdots + W_{L_n}$.
This sum may be controlled by showing that $E_n$ is
bounded with high probability and then bounding the sum by
Chebychev bounds.  The same argument works for $\sum_{i=1}^m M_i$. Toward this end, represent
$$E_n = \sum_{i=1}^m X_i \ \ \ \text{ where } \ \ \ X_i = \begin{cases} 1 & \text{ box $i$ is empty after $n$ balls} \\ 0 & \text{ box $i$ is not empty after $n$ balls} \end{cases}.$$
$$\E\( E_n\) = m \( 1- \frac{1}{m}\)^n, \ \ \ \VAR(E_n) = m \( 1 - \frac{1}{m}\)^n +
  m (m-1) \( 1- \frac{1}{m}\)^n - m^2 \( 1- \frac{1}{m}\)^{2n}.$$
By elementary estimates
\begin{equation}\label{eqn:4.1}
\E\( E_n\) = 1 + O\(\frac{C}{\sqrt{n}}\), \ \ \ \VAR(E_n) = 1 + O\(\frac{C}{\sqrt{n}}\).
\end{equation}
Indeed,
$m \( 1- \frac{1}{m}\)^n = e^{\log(m) - \frac{n}{m} + O\( \frac{n}{m^2}\)}$.
Using the assumption $m = \frac{n}{\alpha_n}
 + \frac{c\sqrt{n}}{\log(n)}$, $\log(m) = \alpha_n + O\( \frac{c}{\sqrt{n}}\),
 \frac{n}{m}  = \alpha_n + O\( \frac{1}{\sqrt{n} \log(n)}\).$  This gives the first result in
 \eqref{eqn:4.1}, the second follows similarly. By classical results \cite{CDM}, $E_n$ is approximately
 POISSON(1) distributed with an explicit total variation error but this is not needed.

Consider next
$$S_n = W_1 + \cdots + W_m = m W_1 + (m-1) (W_2 - W_1) + \cdots + 2 (W_{m-1} - W_{m-2})
+ (W_{m-1} - W_{m}).$$
\begin{align}
&\E(S_n) = \frac{m}{1} + \frac{m-1}{\frac{m-1}{m}} + \cdots + \frac{1}{\frac{1}{m}} = m^2 \\
&\VAR(S_n) = \sum_{i=1}^{m-1} (m-i)^2 \frac{m}{m-i} \( \frac{m}{m-i} -1\) = \sum_{i=1}^{m} m i = m\frac{m(m-1)}{2} \sim \frac{m^3}{2}.
\end{align}

Consider next the sum of the box maxima.  Drop balls labelled $n, n-1, \cdots, 1$ sequentially into $m$ boxes.
If the new arrivals are at times $\wt{W_1}, \wt{W_2}, \cdots, \wt{W_m}$, the box maxima are
$n-(\wt{W_1}-1), n - (\wt{W_2} - 1), \ldots, n - (\wt{W_m}-1)$.  The sum
$$\wt{S_n} = \sum_{i=1}^m m_i = nm - \( \wt{W_1} + \cdots + \wt{W_m}\) + n.$$
Thus
\begin{align}
&\E(\wt{S_n}) = n (m+1) - m^2 \\
&\VAR(\wt{S_n}) = m \frac{m(m-1)}{2}.
\end{align}

The random variable of interest is
$$D_n = \sum_{i=1}^m (M_i - m_i +1) = \wt{S_n} - S_n - \sum_{i=L_n + 1}^m \( \wt{W_i} - W_i\) +m.$$

The sum $\sum_{i=L_n+1}^m \wt{W_i} \le E_n \wt{W_m}.$ From the coupon collectors problem
$\wt{W_m}$ is of stochastic order $m\log(m) \sim n$ and $E_n$ is stochastically
bounded. A similar argument holds with $\wt{W_i}$ replaced by $W_i$.  It follows that the sum
$\sum_{i=L_n +1}^m \( \wt{W_i} - W_i\) = O_p(n).$  Combining terms
$$D_n = \wt{S_n} - \wt{S_n} + m + O_p(n) = nm  + \( \wt{S_n} - \E(\wt{S_n})\) - \( S_n - \E(S_n)\) + O_p(n).$$
By Chebychev's inequality
$\abs{S_n - \E(S_n)}$ and $\abs{\wt{S_n} - \E(\wt{S_n})}$ are both $O_p(m^{3/2})$. It follows that
$D_n = nm -2m^2 + O_p(m^{3/2})$.
\end{proof}

\begin{proof}[Proof of Theorem \ref{thm:dimension}]

To finish the proof of Theorem \ref{thm:dimension} note that conditional on $M$
$$
d(\lambda) = nM -2M^2+O_p(M^{3/2}).
$$
This is weak star close  to
$$
N\left( \frac{n^2}{\alpha_n}-\frac{2n^2}{\alpha_n^2},\frac{n^{3/2}}{\alpha_n}\right) + O_p\left(\frac{n}{\alpha_n} \right)^{3/2}.
$$
Since $(n/\alpha_n)^{3/2}$ is much smaller than the standard deviation of the normal, this is in turn close to
$$
N\left( \frac{n^2}{\alpha_n}-\frac{2n^2}{\alpha^2},\frac{n^{3/2}}{\alpha_n}\right).
$$
This completes the proof.
\end{proof}

\section{Proof of Theorem \ref{thm:crossing}}\label{sec:crossing}

This section contains the proof of Theorem \ref{thm:crossing}. Our approach is to compare the crossing statistic to the dimension statistic, which by Theorem \ref{thm:dimension} is known to be normally distributed.

\begin{proof}
To analyze the distribution of the crossing number, we compare it to the dimension index. We do this by producing a uniform random set partition $\lambda$ in the following unusual way:
\begin{itemize}
\item Pick $M$ from $\mu_n$.
\item Pick a uniform random set partition $\mu$ for that $m$ according to Stam's algorithm.
\item Let $\lambda$ be a uniform random set partition conditional on the event that the set of minimum elements of blocks of $\lambda$ is the set of minimum elements of blocks on $\mu$ and that the set of maximum elements of blocks of $\lambda$ equals the set of maximum elements of blocks of $\mu$.
\end{itemize}
This third step can be accomplished in the following way, assigning the elements of $[n]$ to blocks in order. We begin with no blocks and add elements to blocks one at a time, sometimes creating new blocks. If an element $k$, where $k$ is the maximum element of some block of $\mu$ is added to a block in $\lambda$, we declare that block \emph{closed}. After having assigned the first $k$ elements to blocks in $\lambda$, we assign $k+1$ to a uniform random un-closed block, unless $k+1$ is the minimum element of some block of $\mu$, in which case we assign $k+1$ to a new block of $\lambda$. This procedure clearly produces a uniform $\lambda$ subject to the restriction on the minimum and maximum elements of blocks.

On the other hand, this method of choosing $\lambda$ gives us a reasonable way to analyze $cr(\lambda)$. In particular, the crossing number of $\lambda$ equals the number of pairs of a $j\in[n]$ and a block $B$ in $\lambda$ with
\begin{itemize}
\item $j\not\in B$
\item $j$ not the first element of its block
\item $\max(B) > j$
\item The element of $B$ immediately preceding $j$ is larger than the element of $j$'s block immediately preceding $j$
\end{itemize}
We note that this is easy to analyze given the procedure above for choosing $\lambda$. Suppose that when $k$ is being added to $\lambda$ that there are $a_k$ blocks of $\lambda$ currently open. If $k$ is the first element of its block, then we have no crossings with $j=k$. Otherwise, we claim that the number of crossings with $j=k$ (which we call $X_k$) has distribution given by the discrete uniform random variable on $[0,a_k-1]$. In particular, if the open blocks are $B_1,\ldots,B_{a_k}$ whose element immediately preceding $k$ is $m_1<m_2<\ldots<m_{a_k}$, then $X_k = k-i$ if $k$ is assigned to block $B_i$. Note furthermore, that the $a_k$ are determined by $\mu$ and that the $X_k$ are independent conditional on $\mu$. Since $cr(\lambda)=\sum_k X_k$ is a sum of independent random variables, it is easy to see that conditioned on $\mu$ that with high probability $cr(\lambda)$ is
weak star close to
$$
N\left(\sum_{k\textrm{ not a minimum}}\frac{a_k-1}{2},\sqrt{\sum_{k\textrm{ not a minimum}}\frac{a_k^2-1}{12}} \right).
$$
We note that a given block contributes to $a_k$ if and only if $k$ is between is minimum and maximum values. Therefore,
$$
\sum_{k=1}^n (a_k-1) = \left(\sum_{i=1}^m M_i-m_i\right)-n = nm-2m^2 + O_p(m^{3/2}).
$$
On the other hand, the sum over $a_k$ at the start of blocks is the number pairs of blocks that overlap. Note that for $m = n/\alpha_n +o_p(n/\log^2(n))$, that any given block has $n/2$ between its minimum and maximum with probability $1-O(m^{-1/2})$. Thus, for $m$ in this range, the expected number of pairs of non-overlapping blocks is $O(m^{3/2})$. Thus,
$$
\sum_{k\textrm{ not a minimum}}\frac{a_k-1}{2} = nm/2 - 5m^2/4 + O_p(m^{3/2}).
$$
It is also easy to see that
$$
\sum_{k\textrm{ not a minimum}}(a_k^2-1) = n^2m(1+o_p(1)) = \frac{n^3}{\alpha_n}(1+o_p(1)).
$$
Therefore, with probability approaching $1$ over the choice of $m$, the distribution of $\lambda$
conditioned on $m$ is close to
$$
N\left( nm/2 - 5m^2/4, \frac{n^{3/2}}{\sqrt{12\alpha_n}}\right).
$$
This can be rewritten (up to small error) as the sum of $(n/2+5n/(2\alpha_n))(m-n/\alpha_n)$ and a variable with distribution
$$
N\left( \frac{n^2}{2\alpha_n} - \frac{5n^2}{4\alpha_n^2}, \frac{n^{3/2}}{\sqrt{12}\alpha_n}\right).
$$
On the other hand, by Theorem \ref{thm:hwang}, $(n/2+5n/(2\alpha_n))(m-n/\alpha_n)$ is approximated by an independent normal weak star close to
$$
N\left(0,\frac{n^{3/2}}{2\alpha_n}\right).
$$
Thus, the distribution of $cr(\lambda)$ is close in cdf distance to this sum of independent normals, which is given by
$$
N\left( \frac{n^2}{2\alpha_n} - \frac{5n^2}{4\alpha_n^2}, \frac{n^{3/2}}{\sqrt{3}\alpha_n}\right).
$$
This completes the proof.
\end{proof}


\end{document}